\documentclass[11pt]{amsart}
\usepackage[margin=2.5cm]{geometry}
\usepackage{amsfonts, amstext, amsmath, amssymb}
\usepackage{pstricks}

\newtheorem{lemma}{Lemma}
\newtheorem{theorem}[lemma]{Theorem}
\newtheorem{proposition}[lemma]{Proposition}
\newtheorem{corollary}[lemma]{Corollary}

\begin{document}

\title{\textsf{Kakutani's fixed point theorem in constructive mathematics}}

\author[Hendtlass]{Matthew Hendtlass}
 \address{School of Mathematics and Statistics,
    University of Canterbury,
    Christchurch 8041,
    New Zealand}
\email{matthew.hendtlass@canterbury.ac.nz}

\begin{abstract}
\noindent
\textsf{In this paper we consider Kakutani's extension of the Brouwer fixed point theorem within the framework of Bishop's constructive mathematics. Kakutani's fixed point theorem is classically equivalent to Brouwer's fixed point theorem. The constructive proof of (an approximate) Brouwer's fixed point theorem relies on a finite combinatorial argument; consequently we must restrict our attention to uniformly continuous functions. Since Brouwer's fixed point theorem is a special case of Kakutani's, the mappings in any constructive version of Kakutani's fixed point theorem must also satisfy some form of uniform continuity. We discuss the difficulties involved in finding an appropriate notion of uniform continuity, before giving a constructive version of Kakutani's fixed point theorem which is classically equivalent to the standard formulation. We also consider the reverse constructive mathematics of Kakutani's fixed point theorem, and provide a constructive proof of Kakutani's first application of his theorem---a generalisation of von Neumann's minimax theorem. }
\end{abstract}%


\maketitle

\section{Introduction}

\noindent
In this paper we give a constructive treatment of Kakutani's generalisation \cite{kak} of the Brouwer fixed point theorem.

\bigskip
\noindent
Exactly what do we mean by `constructive mathematics'? This paper is set in the rigorous but informal framework of Bishop's constructive mathematics (\textbf{BISH}). Bishop's constructive mathematics is essentially mathematics with intuitionistic logic and dependent choice\footnote{For an introduction to the practice of Bishop's constructive mathematics see \cite{BB,BR,BV}; see \cite{AR,B-MST,F,M} 
for a constructive alternative to ZFC.}. Working with intuitionistic logic ensures that proofs proceed in a manner which preserves computational content: a proof of $A\Rightarrow B$ converts a witness of $A$ into a witness of $B$. In particular, a constructive proof of $\exists_xP(x)$ embodies an algorithm for the construction of an object $x$, and an algorithm verifying that $P(x)$ holds. In this manner, constructive mathematics can be viewed as a high level programming language.

\bigskip
\noindent
The key aspect: Any proof in Bishop's constructive mathematics contains an algorithm which `implements' the result it proves. So if the conclusion of a theorem has some computational content, then so must the hypothesis: we cannot get something for nothing! To get constructively meaningful results, it is therefore necessary to rewrite many classical definitions in a positive way; we must also at times make explicit conditions which hold trivially in classical mathematics, but which may fail in our framework. At still other times, a classical definition given a constructive reading becomes a far stronger property, so here we must adopt a computationally weaker, though classically equivalent alternative. 

\bigskip
\noindent
A set $S$ is said to be \emph{inhabited} if there exists $x$ such that $x\in S$. An inhabited subset $S$ of a metric space $X$ is \emph{located} if for each $x\in X$ the \emph{distance} 
\begin{equation*}
\rho \left( x,S\right) =\inf \left\{ \rho (x,s):s\in S\right\}
\end{equation*}%
\noindent from $x$ to $S$ exists. An $\varepsilon $%
\emph{-approximation} to $S$ is a subset $T$ of $S$ such
that for each $s\in S$, there exists $t\in T$ such that $\rho
(s,t)<\varepsilon $. We say that $S$ is \emph{totally bounded} if for each $%
\varepsilon >0$ there exists a finitely enumerable\footnote{A set is \emph{finitely enumerable} if it is the image of $\{1,\ldots,n\}$ for some $n\in N$, and a set is \emph{finite} if it is in bijection with $\{1,\ldots,n\}$ for some $n\in N$; constructively these notions are distinct.} $\varepsilon $-approximation to $S$. A metric space $X$ is \emph{compact} if it is complete and totally bounded. We will use $\rho$ to represent any metric, and we write $B(x,r)$ for the ball centered on $x$ with radius $r$.

\bigskip
\noindent
Let $U$ be a function from a metric space $X$ into the class $\mathcal{P}^*(X)$ of nonempty subsets of $X$; $U$ is said to be a \emph{set valued mapping on} $X$. We say that $U$ is \emph{convex} (\emph{compact}, \emph{closed}, etc.) if $U(x)$ is convex (compact, closed, etc.) for each $x\in X$. A mapping $U:X\rightarrow\mathcal{P}^*(X)$ is said to be \emph{sequentially upper hemi-continuous} if for each pair of sequences $\left(x_n\right)_{n\geqslant1}$, $\left(y_n\right)_{n\geqslant1}$ in $X$ converging to points $x,y$ in $X$ respectively, if $y_n\in U\left(x_n\right)$ for each $n$, then $\rho(y,U(x))=0$; in particular, if $U$ is closed, then $y\in U(x)$. If $U$ is closed, then $U$ is sequentially upper hemi-continuous if and only if the \emph{graph} 
 $$
  G(U)=\bigcup_{x\in X}\{x\}\times U(x)
 $$
\noindent
of $U$ is closed. A point $x\in S$ such that $x\in U(x)$ is called a \emph{fixed point of} $U$. Kakutani's fixed point theorem is the following.

\begin{quote}
\textbf{Kakutani's fixed point theorem} Let $S$ be a compact convex subset of $\mathbf{R}^n$ and let $U:S\rightarrow\mathcal{P}^*(S) $ be a closed, convex, sequentially upper hemi-continuous mapping. Then $U$ has a fixed point.
\end{quote}

\bigskip
\noindent
If $f$ is sequentially continuous, then $U(x)=\{f(x)\}$ is sequentially upper hemi-continuous; this shows that Kakutani's fixed point theorem is a generalisation of Brouwer's fixed point theorem. Since Brouwer's fixed point theorem (for sequential,\footnote{For uniformly continuous functions, this follows from \textbf{WKL} (see below)---which is equivalent to \textbf{LLPO}---and the constructive approximate Brouwer fixed point theorem. For a sequentially continuous function $f$, we first approximate $f$ by a sequence $\left(f_n\right)_{n\geqslant1}$ of affine, and therefore uniformly continuous, functions to which, given \textbf{LLPO}, we apply Brouwer's fixed point theorem for uniformly continuous functions to produce a sequence of points $\left(x_n\right)_{n\geqslant1}$ such that $x_n=f_n(x_n)$ for each $n$. By \textbf{WKL}, we may suppose that $\left(x_n\right)_{n\geqslant1}$ converges to some point $x$. The sequential continuity of $f$ then ensures that $x$ is a fixed point of $f$.} pointwise, or uniformly continuous functions) is equivalent to the lesser limited principle of omniscience,
  \begin{quote}
   \textbf{LLPO}: For each binary sequence $\left(a_n\right)_{n\geqslant1}$ with at most one non-zero term, either $a_n=0$ for all even $n$ or $a_n=0$ for all odd $n$,
 \end{quote}
\noindent
Kakutani's fixed point theorem implies \textbf{LLPO}, and hence is not constructively valid.

\bigskip
\noindent
Classically, the Kakutani fixed point theorem is equivalent to the Brouwer fixed point theorem (in the sense that it is straightforward to prove one given the other). The constructive proof of (an approximate) Brouwer fixed point theorem, for a uniformly continuous function $f$ from $[0,1]^n$ into itself, uses a combinatorial argument to show that for all $\delta,\varepsilon>0$ either there exists $x,y\in[0,1]^n$ such that $\rho(x,y)<\delta$ and $\rho(f(x),f(y))>\varepsilon$, or we can construct $x\in[0,1]^n$ such that $\rho(x,f(x))<\varepsilon$. Given $\varepsilon>0$, the former possibility is then ruled out by using the uniform continuity of $f$ to choose an appropriate $\delta$. It is clear that $f$ must satisfy some uniform form of continuity\footnote{This in fact only requires the constructively weaker notion of sequential uniform continuity (see \cite{FPT}).} for this approach to work. If we are to prove a constructive version of Kakutani's fixed point theorem using the constructive approximate Brouwer fixed point theorem, we must therefore have our set valued mappings satisfy some form of uniform continuity. 

\bigskip
\noindent
In the next section we discuss the difficulties of formulating an appropriate notion of uniform continuity for set valued mappings, and hence a constructive version of the Kakutani fixed point theorem. We then give our constructive version of Kakutani's fixed point theorem; this result has, classically, both weaker hypothesis and a weaker conclusion---we only construct approximate fixed points---than the classical version, but is classically equivalent to the standard formulation. The final section gives a constructive proof of Kakutani's first application of his theorem: a generalisation of von Neumann's minimax theorem.

\section{Kakutani's fixed point theorem}
\subsection{The basic result}

\noindent
Our first question is: what is the constructive content of the standard classical proofs of Kakutani's fixed point theorem?

\bigskip
\noindent
As we saw before, Kakutani's fixed point theorem implies \textbf{LLPO}; thus any classical proof of Kakutani's fixed point theorem must be non-constructive. On the other hand, Kakutani's original proof of his theorem only requires (several applications of) weak K\"{o}nig's lemma,
 \begin{quote}
  \textbf{WKL} Every infinite tree has an infinite path.
 \end{quote}
\noindent
in addition to constructive methods. Since \textbf{LLPO} and \textbf{WKL} are constructively equivalent (see \cite{Ish90}---this requires a weak form of countable choice), we have the following result.

\begin{theorem}
\label{Kak_LLPO}
Kakutani's fixed point theorem is equivalent to \textbf{LLPO}.
\end{theorem}

\begin{proof}
Suppose \textbf{LLPO} holds. We give Kakutani's original proof, adapted to the unit hypercube. Let $U$ be a sequentially upper hemi-continuous set valued mapping on $[0,1]^n$. For each $k\in\mathbf{N}$ let $f_k$ be the affine extension of a function on
 $$
  \{0,1/k,\ldots,1\}^n
 $$
\noindent
which takes values in $U(x)$ for each $x$ in its domain. Since \textbf{LLPO} implies Brouwer's fixed point theorem, we can construct a sequence $\left(x_k\right)_{k\geqslant1}$ such that $x_k=f_k\left(x_k\right)$ for each $k\in\mathbf{N}$; by \textbf{WKL} we may suppose that $\left(x_k\right)_{k\geqslant1}$ converges to some point $x_0\in[0,1]^n$. Since \textbf{LLPO} allows us to decide whether $a\leqslant0$ or $a\geqslant0$ for each $a\in\mathbf{R}$, for each $k\in\mathbf{N}$ there exists $s\in S_k$ such that
 $$
  x_k\in \{x\in[0,1]^n: 0\leqslant x_i-s_i\leqslant 1/k,1\leqslant i\leqslant n\}\equiv\{x_1^k,\ldots,x_{2^n}^k\}.
 $$
\noindent
Let $\lambda^k_1,\ldots,\lambda^k_{2^n}$ ($k\in\mathbf{N}$) be such that $\lambda_i^k\geqslant0$ for each $i$, $\sum_{i=1}^{2^n}\lambda_i^k=1$, and
 $$
  x_k=\sum_{i=1}^{2^k}\lambda_i^kx_i^k.
 $$
\noindent
Set $y_i^k=f_k(x_i^k)$ for each $k\in\mathbf{N}, 1\leqslant i\leqslant 2^k$. Then $y_i^k\in U(x_i^k)$ for all $i,k$ and 
 $$
  x_k=f_k(x_k)=\sum_{i=1}^{2^k}\lambda_i^ky_i^k.
 $$
\noindent
Applying \textbf{WKL} repeatedly, we may assume that for each $1\leqslant i\leqslant 2^n$ there exist sequences $\left(\lambda_i^k\right)_{k\geqslant1},$ $\left(y_i^k\right)_{k\geqslant1}$ such that $\left(\lambda_i^k\right)_{k\geqslant1}$ converges to $\lambda_i^0$ in $\mathbf{R}$ and $\left(y_i^k\right)_{k\geqslant1}$ converges to $y_i^0$ in $[0,1]^n$. Then $\lambda_i^0\geqslant0$ for each $i$, $\sum_{i=1}^{2^n}\lambda^0_i=1$, and 
 $$
  x_0=\sum_{i=0}^{2^n}\lambda_i^0 y_i^0.
 $$
\noindent
Moreover, $x_i^k\rightarrow x_0$, $y_i^k\rightarrow y_i^0$, and $y_i^k\in U(x_i^k)$ for each $i$; whence, since $U$ is closed and upper hemi-continuous, $y_i^0\in U(x_0)$ for each $i$. It now follows from the convexity of $U$ that 
 $$
  x_0=\sum_{i=0}^{2^n}\lambda_i^0 y_i^0\in U(x_0);
 $$
\noindent
that is, $x_0$ is a fixed point of $U$.

\bigskip
\noindent
We extend this to a closed convex sequentially upper hemi-continuous mapping $U$ on an arbitrary convex compact subset $S$ of $\mathbf{R}^n$ as follows. Let $Q$ be the uniformly continuous function from $\mathbf{R}^n$ into $S$ which takes a point $x$ of $\mathbf{R}^n$ to the (unique) closest point to $x$ in $S$; this function exists by Theorem 6 of \cite{BRS}, and is shown to be uniformly continuous in \cite{FPT} (Theorem 3). We may suppose, without loss of generality, that $S$ is contained in the unit hypercube. Define a set valued mapping $U^\prime$ on the unit hypercube by setting
 $$
  U^\prime(x)=U(Q(x)).
 $$
\noindent
It is easy to see that $U^\prime$ is closed, convex, and sequentially upper hemi-continuous; whence, since $S$ is closed and $U^\prime$ maps into $S$, there exists $x\in S$ such that $x\in U^\prime(x)$. Since $x=Q(x)$, $x$ is also a fixed point of $U$.
\end{proof}

\bigskip
\noindent 
Intuition may suggest that the functions $\left(f_k\right)_{k\geqslant1}$ become closer and closer to $U$ in some sense, and hence that Kakutani's proof contains a proof of the existence of approximate fixed points---a set valued mapping $U$ on a metric space $X$ has \emph{approximate fixed points} if for each $\varepsilon>0$ there exists $x\in X$ such that $\rho(x,U(x))<\varepsilon$; such an $x$ is called an $\varepsilon$-\emph{fixed point of} $U$. However, in order to construct approximate fixed points, we must be able to quantify the `convergence' of these affine approximations. Our eventual solution to finding a constructive Kakutani fixed point theorem is to restrict our mapping $U$ in such a way as to ensure that for each $\varepsilon>0$ there exists an affine function contained in an $\varepsilon$-expansion of the graph of $U$.

\bigskip
\noindent
We are ready to begin our journey toward a constructive Kakutani fixed point theorem. The constructive treatment of Brouwer's fixed point theorem suggests that we take the following route:
 \begin{itemize}
  \item[(i)] we should recast upper sequential hemi-continuity as a pointwise property;
  \item[(ii)] we should further consider a uniform notion of upper sequential hemi-continuity;
  \item[(iii)] we should focus on approximate fixed points.
 \end{itemize}
\noindent
It is also natural to insist that the image of each point be a located set; but, working constructively, this severly restricts the set valued mappings we can define. For example, in order to prove one direction of Proposition \ref{LPO}, we have in mind the function $U:[0,1]\rightarrow\mathcal{P}^*([0,1])$ given by
 $$
  U(x)=\left\{\begin{array}{lll}
               \{0\} & & x<1/2\\
               {\white }[0,1] & & x=1/2\\
               \{1\} & & x>1/2.
              \end{array}\right.
 $$
\noindent
However, $U$ is only defined on the subset\footnote{The author would like to thank the referee for pointing this out.} $[0,1/2)\cup\{1/2\}\cup(1/2,1]$ of $[0,1]$, which equals $[0,1]$ only in the presence of the highly nonconstructive limited principle of omniscience
 \begin{quote}
  \textbf{LPO}: For every binary sequence $\left(a_n\right)_{n\geqslant1}$ either $a_n=0$ for all $n$, or there exists $n$ such that $a_n=1$.
 \end{quote}
\noindent
We can overcome this problem by defining a set valued mapping from its graph: set $G$ to be the subset of $[0,1]^2$ given by $[(0,0),(1/2,0)]\cup[(1/2,0),(1/2,1)]\cup[(1/2,1),1,1)]$---where $[x,y]=\{tx+(1-t)y:t\in[0,1]\}$---and let $U(x)=\{y\in[0,1]:(x,y)\in G\}$.

\begin{figure}[htp]
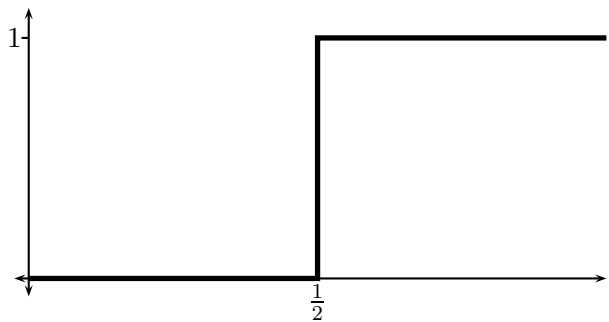

\begin{center}
\pspicture(1,-.5)(6.8,4)
\psset{unit=.8}
\psset{xunit=1.2}

\psline{<->}(-.2,0)(8,0)
\psline{<->}(0,-.3)(0,4.5)

\psline(0,4)(-.1,4)

\psline[linewidth=2pt](0,0)(4.0,0)(4,4.0)(8,4)

\rput[r](-.1,4){$1$}
\rput[t](4,-.1){$\frac{1}{2}$}

\endpspicture
\end{center}
\caption{The graph $G$ of $U:[0,1]\rightarrow\mathcal{P}^*([0,1])$.}\label{fig:U}
\end{figure}

\noindent
Note that in general $U(x)$ need not be located or even inhabited, but it is nonempty.

\bigskip
\noindent
We say that a mapping $U:X\rightarrow\mathcal{P}^*(X)$ is \emph{pointwise upper hemi-continuous} if for each $x\in X$ and each $\varepsilon>0$ there exists $\delta>0$ such that for all $y\in X$, if $\Vert x-y\Vert<\delta$, then
 $$
  U(y)\subset \left(U(x)\right)_\varepsilon,
 $$
\noindent
where, for a subset $S$ of a metric space $X$ and a positive real number $\varepsilon$,
 $$
  S_\varepsilon=\left\{x\in X: \exists_{s\in S}\rho(x,s)<\varepsilon\right\}.
 $$
 
\bigskip
\noindent  
If $U$ is pointwise upper hemi-continuous, then $U$ is sequentially upper hemi-continuous. Suppose that $U$ is a pointwise upper hemi-continuous function and let $\left(x_n\right)_{n\geqslant1}$, $\left(y_n\right)_{n\geqslant1}$ be sequences in $X$ converging to points $x$ and $y$ of $X$, respectively, such that $y_n\in U\left(x_n\right)$ for each $n$. Fix $\varepsilon>0$ and let $\delta>0$ be such that $U(z)\subset \left(U(x)\right)_{\varepsilon/2}$ for all $z\in B(x,\delta)$. Pick $N>0$ such that $\rho(x_n,x)<\delta$ and $\rho(y_n,y)<\varepsilon/2$ for all $n\geqslant N$. Then for each $n\geqslant N$ we have
 $$
  y_n\in U(x_n)\subset (U(x))_{\varepsilon/2},
 $$
\noindent
so
\begin{eqnarray*}
 \rho(y,U(x)) & \leqslant & \rho(y,y_n)+\rho(y_n,U(x))\\
    & < & \varepsilon/2 + \varepsilon/2 \ = \ \varepsilon.
\end{eqnarray*}
\noindent
Since $\varepsilon>0$ is arbitrary, $\rho(y,U(x))=0$.

\begin{proposition}
\label{LPO}
 The statement
 \begin{quote}
  Every sequentially upper hemi-continuous mapping with a separable graph is pointwise upper hemi-continuous.
 \end{quote}
\noindent
is equivalent to \textbf{LPO}.
\end{proposition}

\begin{proof}
 Let $U$ be a sequentially upper hemi-continuous mapping on $X$, let $\left(\left(x_n,y_n\right)\right)_{n\geqslant1}$ be a dense sequence in the graph of $U$, and fix $x\in X$ and $\varepsilon>0$. Using \textbf{LPO}, construct a binary double sequence $\left(\lambda_{k,n}\right)_{k,n\geqslant1}$ such that
 $$
  \lambda_{k,n}=1\ \ \Leftrightarrow\ \ \rho\left(x,x_n\right)<\frac{1}{k}\wedge\rho\left(y_n,U(x)\right)\geqslant\varepsilon.
 $$
\noindent
By \textbf{LPO}, either for all $k$ there exists an $n\in\mathbf{N}$ such that $\lambda_{n,k}=1$, or else there exists $k\in\mathbf{N}$ such that $\lambda_{k,n}=0$ for all $n\in\mathbf{N}$. If there exists $k\in\mathbf{N}$ such that $\lambda_{k,n}=0$ for each $n\in\mathbf{N}$, then $\delta=1/k$ satisfies the definition of pointwise upper hemi-continuity and we are done. Therefore it suffices to rule out the former case: if for each $k$ there exists $n$ such that $\lambda_{k,n}=1$, then there exist sequences $\left(x_n\right)_{n\geqslant1}, \left(y_n\right)_{n\geqslant1}$ in $X$ such that $\left(x_n\right)_{n\geqslant1}$ converges to $x$, and for each $n$, $y_n\in U\left(x_n\right)$ and $\rho\left(y_n,U(x)\right)\geqslant\varepsilon$---a contradiction.

\bigskip
\noindent
To show the converse consider the function $U:[0,1]\rightarrow\mathcal{P}^*([0,1])$ pictured in Figure \ref{fig:U}. It is straightforward to show that $G(U)$ is closed and hence that $U$ is sequentially upper hemi-continuous. Suppose that $U$ is pointwise upper hemi-continuous and let $a$ be a number close to $0$. Let $\delta>0$ be such that if $y\in B(\vert a\vert+1/2,\delta)$, then
 $$
  U(y)\subset \left(U(\vert a\vert+1/2)\right)_{1/2}.
 $$
\noindent
Either $a\neq0$ or $\vert a\vert<\delta$. In the latter case, if $a\neq0$, then 
 $$
  [0,1]=U(1/2)\subset U(\vert a\vert+1/2)_{1/2}=[1/2,1],
 $$
\noindent
which is absurd. Hence $a$ is in fact equal to $0$. Thus the sequential upper hemi-continuity of $f$ implies 
 $$
  \forall_{a\in\mathbf{R}}\left(a=0\vee a\neq0\right),
 $$
\noindent
which in turn implies \textbf{LPO}.
\end{proof}

\bigskip
\noindent
The natural notion of uniform pointwise upper hemi-continuity seems to be too strong to be of much interest. In particular, if $U$, in addition to satisfyin the uniform notion of pointwise upper hemi-continuity, is located, then $U$ is uniformly continuous as a function from $X$ to $\mathcal{P}^*(X)$ endowed with the standard Hausdorff metric, defined for located subsets $S,T$ of $X$ by
 $$
  \rho(S,T)=\max\left\{\sup\{\rho(s,T):s\in S\},\sup\{\rho(t,S):t\in T\}\right\}.
 $$
\noindent
The uniform version of pointwise upper hemi-continuity is not classically equivalent to the non-uniform version because, for a fixed $\varepsilon$, the $\delta(x)$ satisfying sequential upper hemi-continuity at $x$ need not vary continuously with $x$ and may fail to be bounded below by a positive valued continuous function. Another result of this is that few functions are constructively pointwise upper hemi-continuous; for instance the (benign) mapping given in the proof of Proposition \ref{LPO}.

\bigskip
\noindent
The uniform continuity of $U$ is a significantly stronger property than that of sequential upper hemi-continuity, and easily leads to an approximate fixed point theorem of relatively little interest.

\bigskip
\noindent
In order to find a more satisfactory constructive version of Kakutani's fixed point theorem (preferably classically equivalent to the classical one), we need to find a notion similar to pointwise sequential upper hemi-continuity, and with more computational content than sequential upper hemi-continuity, for which the uniform version is classically equivalent to the non-uniform version. To that end, we say that a mapping $U:X\rightarrow\mathcal{P}^*(X)$ is \emph{locally approximable} if for each $x\in X$ and each $\varepsilon>0$, there exists $\delta>0$ such that if $y,y^\prime\in B(x,\delta)$, $u\in U(y)$, $u^\prime\in U(y^\prime)$, and $t\in[0,1]$, then 
 $$
  \rho\left(\left(z_t,u_t\right),G(U)\right)<\varepsilon,
 $$
\noindent
where $z_t=ty+(1-t)y^\prime$ and $u_t=tu+(1-t)u^\prime$; note that we do not require $G(U)$ to be located here: we use `$\rho(x,S)<\varepsilon$' as a shorthand for `there exists $s\in S$ such that $\rho(x,s)<\varepsilon$'.

\begin{proposition}
 Every convex, pointwise upper hemi-continuous set valued mapping on a linear metric space is locally approximable.
\end{proposition}

\begin{proof}
Let $X$ be a linear metric space and let $U$ be a convex, pointwise upper hemi-continuous set valued mapping on $X$. Fix $x\in X$ and $\varepsilon>0$, and let $\delta\in(0,\varepsilon/2)$ be such that $U(y)\subset \left(U(x)\right)_{\varepsilon/2}$ for all $y\in B(x,\delta)$. Let $y,y^\prime\in B(x,\delta)$, $u\in U(x)$, $u^\prime\in U(y)$, and $t\in[0,1]$; since $\left(U(x)\right)_{\varepsilon/2}$ is convex, $u_t\in \left(U(x)\right)_{\varepsilon/2}$. Then
 $$
  \rho\left(z_t,x\right)\leqslant\max\{\rho(x,y),\rho(x,y^\prime)\}<\delta<\varepsilon/2,
 $$
\noindent
so
 \begin{eqnarray*}
  \rho\left(\left(z_t,u_t\right),G(U)\right) & \leqslant & \rho\left(\left(z_t,u_t\right),\{x\}\times U(x)\right)\\
       & < & \rho\left(\left(z_t,u_t\right),\left(x,u_t\right)\right)+\varepsilon/2\\
       & = & \rho\left(z_t,x\right)+\varepsilon/2\\
       & < & \varepsilon/2 +\varepsilon/2=\varepsilon.
 \end{eqnarray*}
\noindent
Hence $U$ is locally approximable.
\end{proof}

\bigskip
\noindent
Above we have negotiated from the sequential upper hemi-continuity of a set valued mapping to its being locally approximable via the property of pointwise upper hemi-continuity. By Proposition \ref{LPO} this requires \textbf{LPO}. It seems likely that we can prove this (classical) equivalence directly under the assumption of weaker nonconstructive principles; our next result is a first attempt at this.

\bigskip
\noindent
A subset $S$ of $\mathbf{N}$ is said to be \emph{psuedobounded} if for each sequence $\left(x_n\right)_{n\geqslant1}$ in $S$ there exists a natural number $N$ such that $x_n\leqslant n$ for all $n\geqslant N$. The following boundedness principle was introduced by Ishihara in \cite{ISHJSL2}.

\begin{quote}
 \textbf{BD}-$\mathbf{N}$ Every countable psuedobounded set is bounded.
\end{quote}

\noindent
In \cite{ISHJSL2}, Ishihara showed that \textbf{BD-}$\mathbf{N}$ holds\footnote{See \cite{BR} for an introduction to intuitionistic and recursive mathematics; \textbf{BD-}$\mathbf{N}$ is
provably false in another model of \textbf{BISH} (see \cite{Lietz}).} in the classical,
intuitionistic, and recursive models of \textbf{BISH}, and that the statement `every sequentially continuous mapping from a separable metric space is pointwise continuous' is equivalent, over \textbf{BISH}, to \textbf{BD-}$\mathbf{N}$. The proof of the following uses ideas from \cite{ISHJSL2}.

\begin{proposition}
\label{sUH&app}
 Suppose \textbf{LLPO} and \textbf{BD}-$\mathbf{N}$ hold. Let $U$ be a convex, sequentially upper hemi-continuous set valued mapping on a separable metric space. Then $U$ is locally approximable.
\end{proposition}

\begin{proof}
Let $\left(x_n\right)_{n\geqslant1}$ be a dense sequence in $S$, and fix $x\in S$ and $\varepsilon>0$. Define
 \begin{eqnarray*}
  A & = & \{0\}\cup\left\{ k>0:\exists_{m,n}\left(x_m,x_n\in B(x,k^{-1})\wedge\right.\right. \\
   & &  \ \ \ \ \ \ \ \ \ \ \ \ \ \ \ \ \ \left.\left.\exists_{t\in[0,1]}\exists_{u\in U(x_m)}\exists_{u^\prime\in U(x_n)}(\rho((z_t,u_t),G(U))>\varepsilon/2)\right)\right\},
 \end{eqnarray*}
\noindent
where $z_t=tx_m+(1-t)x_n$ and $u_t=tu+(1-t)u^\prime$. We show that the set $A$ is pseudobounded. It then follows from \textbf{BD}-$\mathbf{N}$ that there exists $M>0$ such that $a<M$ for all $a\in A$. The definition of $A$ then ensures that $\delta=1/M$ satisfies the definition of local approximability.

\bigskip
\noindent
Let $\left(a_n\right)_{n\geqslant1}$ be a nondecreasing sequence in $A$. Using countable choice, we construct a binary sequence $\left(\lambda_n\right)_{i\geqslant1}$ such that
 \begin{eqnarray*}
  \lambda_n=0 & \Rightarrow & a_n/i<1;\\
  \lambda_n=1 & \Rightarrow & a_n/i>1/2.
 \end{eqnarray*}
\noindent
Let $u\in U(x)$, and construct sequences $\left(x_n\right)_{n\geqslant1},\left(x_n^\prime\right)_{n\geqslant1},\left(y_n\right)_{n\geqslant1},\left(y_n^\prime\right)_{n\geqslant1}$ in $S$ and a sequence $\left(t_n\right)_{n\geqslant1}$ in $[0,1]$ as follows. If $\lambda_n=0$, set $x_n=x_n^\prime=x$, $y_n=y_n^\prime=u$, and $t_i=0$. If $\lambda_n=1$, then pick $x_k,x_l\in B(x,i^{-1}), u\in U(x_k),u^\prime\in U(x_l)$, and $t\in[0,1]$ such that $\rho((z_t,u_t),G(U))>\varepsilon$, and set $x_n=x_k$, $x_n^\prime=x_l$, $y_n=u$, $y_n^\prime=u^\prime$, and $t_n=t$. Then $\left(x_n\right)_{n\geqslant1}$ and $\left(x_n^\prime\right)_{n\geqslant1}$ converge to $x$, and, since $\left(a_n\right)_{n\geqslant1}$ is nondecreasing and \textbf{LLPO} implies \textbf{WKL}, we may assume that there exist $y,y^\prime\in S$ and $t\in[0,1]$ such that $\left(y_n\right)_{n\geqslant1}$ converges to $y$, $\left(y_n^\prime\right)_{n\geqslant1}$ converges to $y^\prime$, and $\left(t_n\right)_{n\geqslant1}$ converges to $t$. Then, by the sequential upper hemi-continuity of $U$,  $\rho(y,U(x))=\rho(y^\prime,U(x))=0$; whence, since $U(x)$ is convex, $\rho(ty+(1-t)y^\prime,U(x))=0$. Let $N>0$ be such that
 $$
   \max\left\{\rho((x_n,y_n),(x,y)),\rho((x_n^\prime,y_n^\prime),(x,y))\right\}<\varepsilon
 $$
\noindent
for all $n\geqslant N$, and suppose that there exists $n\geqslant N$ with $\lambda_n=1$. Then 
 $$
  \rho((t_nx_n+(1-t_n)x_n^\prime,t_ny_n+(1-t_n)y_n^\prime),G(U))>\varepsilon,
 $$
\noindent
but 
 $$
  \rho((t_nx_n+(1-t_n)x_n^\prime,t_ny_n+(1-t_n)y_n^\prime),U(x))\ <\  \varepsilon + \rho((x,z),U(x))\ =\ \varepsilon.
 $$
\noindent
This contradiction ensures that $\lambda_n=0$ for all $n\geqslant N$; thus $A$ is pseudobounded.
\end{proof}

\bigskip
\noindent
We say that $U:X\rightarrow\mathcal{P}^*(X)$ is \emph{approximable} if it satisfies the uniform version of local approximability: for each $\varepsilon>0$, there exists $\delta>0$ such that if $x,x^\prime\in X$, $\Vert x-x^\prime\Vert<\delta$, $u\in U(x)$, $u^\prime\in U(x^\prime)$, and $t\in[0,1]$, then 
 $$
  \rho\left(\left(z_t,u_t\right),G(U)\right)<\varepsilon,
 $$
\noindent
where $z_t=tx+(1-t)x^\prime$ and $u_t=tu+(1-t)u^\prime$. Given a function $f:X\rightarrow X$, define $U_f:X\rightarrow\mathcal{P}^*(X)$ by $U_f(x)=\{f(x)\}$. If $f$ is continuous then $U_f$ is locally approximable, and if $f$ is uniformly continuous then $U_f$ is approximable.

\bigskip
\noindent
With the help of Brouwer's fan theorem for $\Pi^0_1$-bars, we can show that every locally approximable function on $[0,1]^n$ is approximable. We will prove a more general result.

\bigskip
\noindent
Let $2^\mathbf{N}$ denote the space of binary sequences, Cantor's space, and let $2^*$ be the set of finite binary sequences. A subset $S$ of $2^*$ is \emph{decidable} if for each $a\in 2^*$ either $a\in S$ or $a\notin S$. For two elements $u=(u_1,\ldots,u_m),v=(v_1,\ldots,v_n)\in 2^*$ we denote by $u\frown v$ the \emph{concatenation}
 $$
  (u_1,\ldots,u_m,v_1,\ldots,v_n)
 $$
\noindent
of $u$ and $v$. For each $\alpha\in 2^\mathbf{N}$ and each $N\in\mathbf{N}$ we denote by $\overline{\alpha}(N)$ the finite binary sequence consisting of the first $N$ terms of $\alpha$. A set $B$ of finite binary sequences is called a \emph{bar} if for each $\alpha\in 2^\mathbf{N}$ there exists $N\in\mathbf{N}$ such that $\overline{\alpha}(N)\in B$. A bar $B$ is said to be \emph{uniform} if there exists $N\in \mathbf{N}$ such that for each $\alpha\in2^\mathbf{N}$ there is $n\leqslant N$ with $\alpha(n)\in B$. The weakest form of Brouwer's fan theorem is:

\begin{quote}
 \textbf{FT$_\Delta$}: Every decidable bar is uniform.
\end{quote}

\noindent
Stronger versions of the fan theorem are obtained by allowing more complex bars. A set $S$ is said to be a $\Pi^0_1$\emph{-set} if there exists a detachable subset $D$ of $2^*\times\mathbf{N}$ such that 
 \begin{quote}
  for each $u\in2^*$ and each $n\in\mathbf{N}$, if $(u,n)\in D$, then $(u\frown0,n)\in D$ and $(u\frown1,n)\in D$,
 \end{quote}
\noindent
and $S=\left\{u\in2^*:\forall_{n\in\mathbf{N}}(u,n)\in D\right\}$. It is easy to see that the fan theorem for bars which are also $\Pi^0_1$-sets
 \begin{quote}
 \textbf{FT}$_{\Pi^0_1}$: Every $\Pi^0_1$-bar is uniform,
\end{quote}
\noindent
implies \textbf{FT}$_\Delta$. Recent work of Hannes Diener and Robert Lubarsky \cite{DL} shows that \textbf{FT}$_{\Pi^0_1}$ is strictly stronger than \textbf{FT}$_\Delta$ over \textbf{IZF}. Brouwer's fan theorem is not intuitionistically valid, but is accepted by some schools of constructive mathematics (see \cite{BR}). 

\bigskip
\noindent
Each fan theorem has an equivalent formulation where we allow an arbitrary finitely branching tree in the place of the binary fan $2^\mathbf{N}$. The fan theorem for a finitely branching tree is reduced to that on $2^\mathbf{N}$ by replacing tree width by tree depth: a tree consisting of a root with $n$ branches can be treated as a binary tree with depth $\lceil\log_2(n)\rceil$, possibly with some branches duplicated.

\bigskip
\noindent
A predicate $P$ on $S\times S\times \mathbf{R}^+$ is said to be a \emph{pointwise continuous predicate on} $S$ if
 \begin{itemize}
 \item[(i)] for each $\varepsilon>0$ and each $x\in S$, there exists $\delta>0$ such that if $y,y^\prime\in B(x,\delta)$, then $P(y,y^\prime,\varepsilon)$;
 \item[(ii)] if $\varepsilon>0$, $x^\prime\in S$, and $\left(x_n\right)_{n\geqslant1}$ is a sequence in $S$ converging to a point $x$ of $S$ and such that $P(x_n,x^\prime,\varepsilon)$ for each $n$, then $P(x,x^\prime,\varepsilon)$;
 \item[(iii)]  for all $x,x^\prime\in S$ and each $\varepsilon>0$ either $P(x,x^\prime,\varepsilon)$ or $\neg P(x,x^\prime,\varepsilon/2)$.
\end{itemize}
\noindent
If the $\delta$ in condition (i) can be chosen independent of $x$, then $P$ is said to be a \emph{uniformly continuous predicate on} $S$.

\bigskip
\noindent
For example to each pointwise (resp. uniformly) continuous function $f:S\rightarrow\mathbf{R}$ we can associate a pointwise (resp. uniformly) continuous predicate given by
 $$
  P(x,x^\prime,\varepsilon)\equiv |f(x)-f(x^\prime)|<\varepsilon,
 $$
\noindent
and to each locally approximable (resp. approximable) function we can associate a pointwise (resp. uniformly) continuous predicate $P$ such that $P(x,x^\prime,\delta)$ if and only if for all $u\in U(x)$, $u^\prime\in U(x^\prime)$, and $t\in[0,1]$, 
 $$
  \rho\left(\left(z_t,u_t\right),G(U)\right)<\varepsilon.
 $$
\noindent
The proof of the following is based on that of Theorem 2 of \cite{HI}. 

\begin{lemma}
\label{UCT1}
 Assume the fan theorem for $\Pi^0_1$-bars. Then every pointwise continuous predicate on $[0,1]$ is uniformly continuous.
\end{lemma}

\begin{proof}
Let $P$ be a pointwise continuous predicate on $[0,1]$ and fix $\varepsilon>0$. Let $X=\{-1,0,1\}$ and let $X^*$ be the set of finite sequences of elements of $X$. Define $f:X^\mathbf{N}\rightarrow[0,1]$ by
 $$
   f(\alpha)=\frac{1}{2} + \sum_{n=1}^\infty 2^{-(n+1)}\alpha(n);
 $$
\noindent
then $f$ is a uniformly continuous function which maps $X^\mathbf{N}$ onto $[0,1]$. Using countable choice, we construct a binary valued function $\lambda$ on $X^*\times X^*$ such that
  \begin{eqnarray*}
    \lambda(u,v)=1 & \Rightarrow & P(f(u\frown\mathbf{0}),f(v\frown\mathbf{0}),\varepsilon);\\
    \lambda(u,v)=0 & \Rightarrow & \neg P(f(u\frown\mathbf{0}),f(v\frown\mathbf{0}),\varepsilon/2),
  \end{eqnarray*}
 \noindent
 where $\mathbf{0}=(0,\ldots)$. Let $D$ be the set of pairs $(u,n)$ in $ X^*\times\mathbf{N}$ such that
  \begin{quote}
   for all $v,w\in X^*$ with lengths at most $n-|u|$, $\lambda(u\frown v,u\frown w)=1$;
  \end{quote}
 \noindent
 clearly $D$ is a decidable set and if $(u,n)\in D$, then $(u\frown a,n)\in D$ for each $a\in\{-1,0,1\}$. Hence 
  $$
   B=\left\{u\in X^*:\forall_{n\in\mathbf{N}}(u,n)\in D\right\}
  $$
\noindent
is a $\Pi^0_1$-set. To see that $B$ is a bar, consider any $\alpha\in X^\mathbf{N}$. Since $P$ is a pointwise continuous predicate, there exists $\delta>0$ such that $P(y,y^\prime,\varepsilon/2)$ for all $y,y^\prime\in B(f(\alpha),\delta)$. If $N>0$ is such that $2^{-N}<\delta$, then for all $u,v\in X^*$
 $$
  f(\overline{\alpha}(N)\frown u\frown\mathbf{0}),f(\overline{\alpha}(N)\frown v\frown\mathbf{0})\in B(x,\delta),
 $$
\noindent
so 
 $$
  P(f(\overline{\alpha}(N)\frown u\frown\mathbf{0}),f(\overline{\alpha}(N)\frown v\frown\mathbf{0}),\varepsilon/2).
 $$
\noindent
Hence $\lambda(\overline{\alpha}(N)\frown u,\overline{\alpha}(N)\frown u)=1$ for all $u,v\in X^*$ and therefore $\overline{\alpha}(N)\in B$.

\bigskip
\noindent
Using \textbf{FT}$_{\Pi^0_1}$, compute $N\in\mathbf{N}$ such that $\overline{\alpha}(N)\in B$ for each $\alpha\in X^\mathbf{N}$. Let $x,y\in[0,1]$ be such that $\rho(x,y)<2^{-(N+1)}$. Then there exist $\alpha,\beta\in X^\mathbf{N}$ such that $f(\alpha)=x$, $f(\beta)=y$, and $\overline{\alpha}(N)=\overline{\beta}(N)$. Since $\overline{\alpha}(N)\in B$, 
 $$
  P(f(\overline{\alpha}(N)\frown u\frown\mathbf{0}),f(\overline{\alpha}(N)\frown v\frown\mathbf{0}),\varepsilon)
 $$
\noindent
for all $u,v\in X^*$. It now follows from condition (ii) of being a pointwise continuous predicate that $P(x,y,\varepsilon)$ holds. Hence $P$ is uniformly continuous.
\end{proof}

\begin{theorem}
 Assume the fan theorem for $\Pi^0_1$-bars. Then every pointwise continuous predicate on $[0,1]^n$ is uniformly continuous.
\end{theorem}

\begin{proof}
 We proceed by induction on $n$. The case $n=1$ is just Lemma \ref{UCT1}. Suppose that the result holds for predicates on $[0,1]^{n-1}$, and let $P$ be a predicate on $[0,1]^n$. For each $x$ in $[0,1]$ let $P_x$ be the predicate on $[0,1]^{n-1}$ given by
 $$
  P_x(z,z^\prime,\varepsilon)\Leftrightarrow P((z,x),(z^\prime,x),\varepsilon).
 $$
\noindent
Since $P$ is a pointwise continuous predicate, $P_x$ is pointwise continuous for each $x\in[0,1]$. It follows from our induction hypothesis that each $P_x$ is uniformly continuous. Define a predicate $P^\prime$ on $[0,1]$ by 
 $$
  P^\prime(s,t,\varepsilon)\ \ \Leftrightarrow \forall_{y\in[0,1]^{n-1}}P_x((s,y),(t,y),\varepsilon).
 $$
\noindent
It is easily shown that $P^\prime$ is also a pointwise continuous predicate and that $P^\prime(s,t,\delta)$ holds for all $s,t\in[0,1]$ if and only if $P(x,x^\prime,\delta)$ holds for all $x,x^\prime\in[0,1]^n$. By Lemma \ref{UCT1}, $P^\prime$ is uniformly continuous; whence $P$ is uniformly continuous.
\end{proof}

\begin{corollary}
\label{FT}
 The fan theorem for $\Pi^0_1$-bars implies that every locally approximable, located mapping on $[0,1]^n$ is approximable.
\end{corollary}

\begin{proof}
 Follows directly from the previous theorem and the statement preceding Lemma \ref{UCT1}.
\end{proof}

\bigskip
\noindent
A mapping $U$ is approximable if for each positive $\varepsilon$ there exists a positive $\delta$ such that the convex hull of any two points in the graph of $U$ which are separated by less than $\delta$ never strays more than $\varepsilon$ from the graph of $U$; our next lemma shows that if $U$ is approximable, then we can generalise this from any two points of $G(U)$ to any finite subset of $G(U)$. This will allow us to give a constructive version of Kakutani's fixed point theorem which is classically equivalent to the classical version.

\begin{lemma}
\label{L_app}
 Let $U:X\rightarrow\mathcal{P}^*(X)$ be an approximable function. Then for each $n>0$ and each $\varepsilon>0$ there exists $\delta>0$ such that for all $x_1,\ldots,x_n,u_1,\ldots,u_n \in X$ and all $\mathbf{t}\in[0,1]^n$, if $u_i\in U(x_i)$ for each $i$, $\sum_{i=1}^nt_i=1$, and 
 $$
  \max\{\Vert x_i-x_j\Vert:1\leqslant i,j\leqslant n\}<\delta,
 $$
\noindent
then 
 $$
  \rho((z_\mathbf{t},u_\mathbf{t}),G(U))<\varepsilon,
 $$
\noindent
where $z_\mathbf{t}=\sum_{i=1}^n t_ix_i$ and $u_\mathbf{t}=\sum_{i=1}^n t_iu_i$.
\end{lemma}

\begin{proof}
We proceed by induction; the case $n=1$ is trivial. Suppose that we have shown the result for $n=k-1$. Let $\mathbf{t}\in[0,1]^k$ and $u_1,\ldots,u_k$ be as in the statement of the lemma, and let $\delta>0$ be such that for all $x_1,\ldots,x_{k-1} \in X^n$ and all $\mathbf{t}\in[0,1]^{k-1}$, if $u_i\in U(x_i)$ for each $i$, $\sum_{i=1}^{k-1}t_i=1$, and 
  $$
   \max\{\Vert x_i-x_j\Vert:1\leqslant i,j\leqslant k-1\}<\delta,
  $$
\noindent
then $\rho((z_\mathbf{t},u_\mathbf{t}),G(U))<\varepsilon/2$. Let $\mathbf{t}^\prime$ be the $k-1$ dimensional vector with $i$-th component $t_i/\sum_{j=1}^{k-1}t_j$. Then
 $$
  \rho\left(\left(z_{\mathbf{t}^\prime},u_{\mathbf{t}^\prime}\right),G(U)\right)<\varepsilon/2.
 $$
\noindent
Picking $(x,u)\in G(U)$ with $\rho\left(\left(z_{\mathbf{t}^\prime},u_{\mathbf{t}^\prime}\right),(x,u)\right)<\varepsilon/2$ and $t\in[0,1]$ such that $\rho((z_\mathbf{t},u_\mathbf{t}),(tx+(1-t)x_n,tu+(1-t)u_n))<\varepsilon/2$, we have that
 \begin{eqnarray*}
  \rho\left(\left(z_{\mathbf{t}},u_{\mathbf{t}}\right),G(U)\right) & < & \rho((z_\mathbf{t},u_\mathbf{t}),(tx+(1-t)x_n,tu+(1-t)u_n)) + \varepsilon/2\\
& < & \varepsilon/2 +\varepsilon/2=\varepsilon.
 \end{eqnarray*}
\noindent
This completes the induction.
\end{proof}

\begin{theorem}
 \label{Kak}
 Let $S$ be a totally bounded subset of $\mathbf{R}^n$ with convex closure and let  $U$ be an approximable set valued mapping on $S$. Then for each $\varepsilon>0$ there exists $x\in S$ such that $\rho(x,U(x))<\varepsilon$.
\end{theorem}

\begin{proof}
Fix $\varepsilon>0$ and let $\delta>0$ be such that for all $x_1,\ldots,x_k \in X^n$ and all $\mathbf{t}\in[0,1]^n$, if $u_i\in U(x_i)$ for each $i$, $\sum_{i=1}^nt_i=1$, and $\max\{\Vert x_i-x_j\Vert:1\leqslant i,j\leqslant k\}<\delta$, then 
 $$
  \rho((z_\mathbf{t},u_\mathbf{t}),G(U))<\varepsilon/3.
 $$

\bigskip
\noindent
Let $S^\prime=\left\{x_1,\ldots,x_l\right\}$ be a discrete $\delta$-approximation to $S$. For each $x_i\in S^\prime$, pick $u_i\in U(x_i)$; let $g$ be the uniformly continuous affine function on $S$ that takes the value $u_i$ at $x_i$ for each $1\leqslant i\leqslant l$. By the approximate Brouwer fixed point theorem, there exists $y\in S$ such that $\rho(y,g(y))<\varepsilon/3$. By our choice of $\delta$, there exists $(x,u)\in G(U)$ such that $\rho((y,g(y)),(x,u))<\varepsilon/3$. Therefore
 \begin{eqnarray*}
  \rho(x,u) & \leqslant & \rho(x,y) +\rho(y, g(y))+\rho(g(y),u)\\
    & < & \varepsilon/3+\varepsilon/3+\varepsilon/3=\varepsilon,
 \end{eqnarray*}
\noindent
so $\rho(x,U(x))<\varepsilon$.
\end{proof}

\bigskip
\noindent
To see that Theorem \ref{Kak} is classically equivalent to the classical theorem let $U:S\rightarrow\mathcal{P}^*(S)$ be as in the classical Kakutani fixed point theorem; then, as previously shown, $U$ is approximable under classical logic. Using the above theorem (and countable choice), construct a sequence $\left(x_n\right)_{n\geqslant1}$ in $X$ such that $\rho(x_n,U(x_n))<1/n$ for each $n$. Since $X$ is compact, we may assume, by \textbf{WKL}, that $\left(x_n\right)_{n\geqslant1}$ converges to some $x\in X$. For each $n$, pick $y_n\in U(x_n)$ such that $\rho(x_n,y_n)<1/n$. Then $y_n\rightarrow x$ and so, since $U$ is closed and sequential upper hemi-continuous, $x\in U(x)$. The following diagram summarises this classical proof.

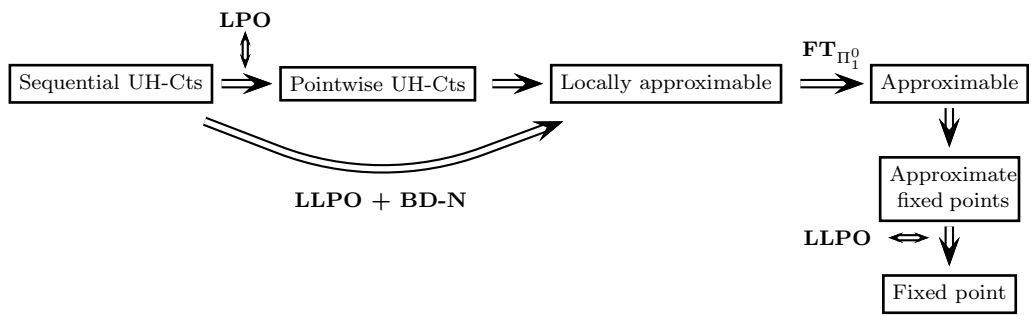
\vspace{-.7cm}

\begin{center}
 \pspicture(0,1)(14,7)
 \psset{xunit=0.9,yunit=1.}

 \rput[l](0,5){\psframebox[linewidth=1pt]{{\scriptsize Sequential UH-Cts}}}
 \rput[l](4,5){\psframebox[linewidth=1pt]{{\scriptsize Pointwise UH-Cts}}}
 \rput[l](8,5){\psframebox[linewidth=1pt]{{\scriptsize Locally approximable}}}
 \rput[c](13.9,5){\psframebox[linewidth=1pt]{{\scriptsize Approximable}}}
 \psline[doubleline=true,doublesep=.07,arrowsize=8pt]{->}(3.15,5)(3.9,5)
 \psline[doubleline=true,doublesep=.07,arrowsize=8pt]{->}(7.15,5)(7.9,5)
 \psline[doubleline=true,doublesep=.07,arrowsize=8pt]{->}(11.7,5)(12.64,5)

 \rput[c](3.5,5.85){{\scriptsize \textbf{LPO}}}
 \psline[doubleline=true,doublesep=.04,arrowsize=4pt]{<->}(3.5,5.2)(3.5,5.7)

 \rput[b](12.17,5.2){{\scriptsize \textbf{FT}$_{\Pi^0_1}$}}

 \psline[doubleline=true,doublesep=.07,arrowsize=8pt,linearc=110pt]{->}(2.9,4.5)(5.5,3.6)(8.2,4.5)

 \rput[t](5.5,3.55){{\scriptsize \textbf{\textbf{LLPO} + \textbf{BD}-$\mathbf{N}$}}}

 \rput(13.9,3.6){\psframebox[linewidth=1pt]{\parbox[c]{1.6cm}
  {\scriptsize Approximate {\white -}fixed points}}}
 \psline[doubleline=true,doublesep=.07,arrowsize=8pt]{->}(13.9,4.675)(13.9,4.15)
 \psline[doubleline=true,doublesep=.07,arrowsize=8pt]{->}(13.9,3.1)(13.9,2.57)

 \rput(13.9,2.2){\psframebox[linewidth=1pt]{{\scriptsize Fixed point}}}

 \psline[doubleline=true,doublesep=.04,arrowsize=4pt]{<->}(13,2.95)(13.6,2.95)
 \rput[r](12.75,2.95){\scriptsize \textbf{LLPO}}

 \endpspicture
\end{center}

\vspace{-1.1cm}

\bigskip
\noindent
In particular, we feel that this gives a conceptually more straightforward classical proof of the Kakutani fixed point theorem than the standard classical proofs: that sequential upper hemi-continuity classically implies approximability is quite intuitive, and that approximable mappings have approximate fixed points is very natural in light of Brouwer's fixed point theorem; it then just remains to apply the result (equivalent to \textbf{WKL}) that any `continuous problem' on a compact space which has approximate solutions has an exact solution.

\subsection{An extension}

\noindent
Theorem \ref{Kak} gives a very simple and intuitive constructive version of Kakutani's fixed point theorem. An examination of the proof shows that we in fact only require our set valued mapping $U$ to satisfy the following condition, weaker than approximability and often much easier to verify. A set valued mapping $U$ on a metric space $X$ is said to be \emph{weakly approximable} if for each $\varepsilon>0$, there exist
 \begin{itemize}
  \item[$\rhd$] a positive real number $\delta<\epsilon$,
  \item[$\rhd$] a $\delta/2$-approximation $S$ of $X$, and
  \item[$\rhd$] a function $V$ from $S$ into $\mathcal{P}^*X$ with $G(V)\subset G(U)$,
 \end{itemize}
\noindent
such that if $x,x^\prime\in S$, $\Vert x-x^\prime\Vert<\delta$, $u\in V(x)$, $u^\prime\in V(x^\prime)$, and $t\in[0,1]$, then 
 $$
  \rho\left(\left(z_t,u_t\right),G(U)\right)<\varepsilon.
 $$
\noindent
If $V$ can be chosen independent of $\varepsilon$, in which case $S$ is a dense subset of $X$, then $U$ is said to be \emph{weakly approximable with respect to} $V$.

\bigskip
\noindent
The proofs of Lemma \ref{L_app} and Theorem \ref{Kak} readily extend to give the following result.

\begin{theorem}
 \label{Kak2}
 Let $S$ be a totally bounded subset of $\mathbf{R}^n$ with convex closure and let  $U$ be a weakly approximable set valued mapping on $S$. Then for each $\varepsilon>0$ there exists $x\in S$ such that $\rho(x,U(x))<\varepsilon$.
\end{theorem}

\section{An application}

\bigskip
\noindent
In \cite{kak}, Kakutani presented his fixed point theorem and used it to give a simple proof of von Neumann's minimax theorem, which guarantees the existence of saddle points for particular functions:

\begin{theorem}
\label{gvN}
 Let $f:[0,1]^n\times [0,1]^m\rightarrow\mathbf{R}$ be a continuous function such that for each $x_0,y_0\in [0,1]$ and each real number $r$  the sets
 \begin{eqnarray*}
 \left\{y\in L:f(x_0,y)\leqslant r\right\}& \mbox{and}\\
 \left\{x\in L:f(x,y_0)\geqslant r\right\}
 \end{eqnarray*}
\noindent
are convex. Then
 $$
  \sup_{x\in [0,1]}\inf_{y\in [0,1]} f(x,y)=\inf_{y\in [0,1]}\sup_{x\in [0,1]}f(x,y).
 $$
\end{theorem}

\bigskip
\noindent
Throughout this section we fix a uniformly continuous function $f:[0,1]^n\times [0,1]^m\rightarrow\mathbf{R}$ satisfying the conditions of Theorem \ref{gvN}, and for each $\varepsilon>0$ we set
 \begin{eqnarray*}
  V_\varepsilon & = & \left\{(x_0,y_0)\in[0,1]^n\times[0,1]^m:f(x_0,y_0)\leqslant\inf_{y\in [0,1]}f(x_0,y)+\varepsilon\right\};\\
  W_\varepsilon & = & \left\{(x_0,y_0)\in[0,1]^n\times[0,1]^m:f(x_0,y_0)\geqslant\sup_{x\in[0,1]}f(x,y_0)-\varepsilon\right\}.
 \end{eqnarray*}
\noindent
In order to prove the minimax theorem, we extend, in the obvious way, the definition of approximable, weakly approximable with respect to, and weakly approximable to functions which take points from a metric space $X$ to subsets of a second metric space $Y$; we call such a function a \emph{set valued mapping from }$X$\emph{ into }$Y$. We associate $V_\varepsilon$ and $W_\varepsilon$ with the set valued mappings given by 
 $$
  V_\varepsilon(x)=\{y\in[0,1]^n:(x,y)\in V\}\mbox{ and }W_\varepsilon(y)=\{x\in[0,1]^n:(y,x)\in W\};
 $$
\noindent
note that $V_\varepsilon,W_\varepsilon$ are convex valued. Let $U_i$ be a set valued mapping from $X_i$ into $Y_i$ ($i=1,2$). The \emph{product} of $U_1$ and $U_2$, written $U_1\times U_2$, is the set valued mapping from $X_1\times X_2$ to $Y_1\times Y_2$ given by
 $$
  U_1\times U_2(x_1,x_2)=U_1(x_1)\times U_2(x_2).
 $$
\noindent
We omit the straightforward proof of the next lemma.

\begin{lemma}
\label{L_minimax1}
Let $U_i$ be a set valued mapping from $X_i$ into $Y_i$ ($i=1,2$). If $U_1,U_2$ are (weakly) approximable, then $U_1\times U_2$ is (weakly) approximable, and if $U_1,U_2$ are weakly approximable with respect to $V_1,V_2$ respectively, then $U_1\times U_2$ is weakly approximable with respect to $V_1\times V_2$.
\end{lemma}

\begin{lemma}
\label{L_minimax2}
For each $\varepsilon>0$, $V_\varepsilon$ is weakly approximable with respect to $V_{\varepsilon/2}$ and $W_\varepsilon$ is weakly approximable with respect to $W_{\varepsilon/2}$. 
\end{lemma}

\begin{proof}
We only give the proof for $V_\varepsilon$; the proof for $W_\varepsilon$ is entirely analogous. Since $f$ is uniformly continuous, there exists $\delta>0$ such that $(V_{\varepsilon/2})_\delta$ is contained in $V_\varepsilon$. Let $x,x^\prime$ be points of $[0,1]^n$ such that $\Vert x-x^\prime\Vert<\delta$ and fix $y,y^\prime$ such that $(x,y),(x^\prime,y^\prime)\in V_{\varepsilon/2}$. Then $(\overline{x},y),(\overline{x},y)\in V_\varepsilon$, where $\overline{x}=(x+x^\prime)/2$. Since $V_\varepsilon(\overline{x})$ is convex valued, $ty+(1-t)y^\prime\in V_\varepsilon$ for each $t\in [0,1]$; whence
 \begin{eqnarray*}
  \rho((\mathbf{z}_t,\mathbf{u}_t),G(u))& \leqslant & \rho((\mathbf{z}_t,\mathbf{u}_t),\{\overline{x}\}\times V_\varepsilon(\overline{x}))\\
  & = & \rho(tx+(1-t)x^\prime,\overline{x}) \ < \ \delta.
 \end{eqnarray*}
\noindent
Since $\delta$ can be chosen to be arbitrarily small, this completes the proof.
\end{proof}

\bigskip
\noindent
We now have the \textbf{proof of Theorem \ref{gvN}}:

\bigskip
\begin{proof}
Let $f:[0,1]\times [0,1]\rightarrow \mathbf{R}$ be as in the statement of the theorem. It is easy to see that 
 $$
  \sup_{x\in [0,1]}\inf_{y\in [0,1]} f(x,y)\leqslant\inf_{y\in [0,1]}\sup_{x\in [0,1]}f(x,y).
 $$
\noindent
Fix $\varepsilon>0$ and let $\delta>0$ be such that $\Vert f(x,y)-f(x^\prime,y^\prime)\Vert<\varepsilon/4$ whenever $\Vert (x,y)-(x^\prime,y^\prime)\Vert<\delta$. By Lemmas \ref{L_minimax1} and \ref{L_minimax2} the set valued mapping $U$ on $[0,1]^{n+m}$ given by
 $$
  U=W_{\varepsilon/2}\times V_{\varepsilon/2}
 $$
\noindent
is approximable with respect to $W_{\varepsilon/4}\times V_{\varepsilon/4}$. By Theorem \ref{Kak2}, there exists $(x_0,y_0)\in[0,1]^{n+m}$ such that $\rho((x_0,y_0),U(x_0,y_0))<\delta$. It follows from the definition of $U$ and our choice of $\delta$, that 
\begin{eqnarray*}
  f(x_0,y_0) & < & \inf_{y\in[0,1]}f(x_0,y)+\varepsilon,\mbox{ and}\\
  f(x_0,y_0) & > & \sup_{x\in[0,1]}f(x,y_0)-\varepsilon.
 \end{eqnarray*}
\noindent
Hence
 \begin{eqnarray*}
  \inf_{y\in [0,1]}\sup_{x\in [0,1]}f(x,y) & \leqslant & \sup_{x\in [0,1]} f(x_0,y_0)\\
    & < & f(x_0,y_0)+\varepsilon\\
    & < & \inf_{y\in[0,1]}f(x_0,y)+2\varepsilon\\ 
    & \leqslant & \sup_{x\in [0,1]}\inf_{y\in [0,1]} f(x,y)+2\varepsilon.
 \end{eqnarray*}
\noindent
Since $\varepsilon>0$ is arbitrary, this completes the proof.
\end{proof}

\end{document}